\documentclass[reqno]{amsart}
\usepackage{amssymb,xy}
\usepackage{amsmath,amscd,amsfonts,amsthm,latexsym}
\usepackage[T1]{fontenc}

\xyoption{all}
\newdir{ >}{!/8pt/\dir{}*\dir{>}}

\newtheorem{theorem}{Theorem}[section]
\newtheorem{proposition}[theorem]{Proposition}
\newtheorem{corollary}[theorem]{Corollary}
\newtheorem{lemma}[theorem]{Lemma}
\theoremstyle{definition}
\newtheorem{definition}[theorem]{Definition}
\newtheorem{example}[theorem]{Example}

\newtheorem{remark}[theorem]{Remark}

\numberwithin{equation}{section}

\def\Ker{\operatorname{Ker}}

\def\Hom{\operatorname{Hom}}

\def\lim{\operatorname{lim}}

\DeclareMathOperator{\cl}{cl}
\DeclareMathOperator{\Brnr}{Br_{nr}}
\DeclareMathOperator{\Gab}{G_{ab}}

\begin{document}

\title[ Non-abelian tensor product and the Bogomolov multiplier]{On some closure properties of the non-abelian tensor product and the Bogomolov multiplier}
\author[G. Donadze, M. Ladra, V. Thomas]{G. Donadze, M. Ladra, V. Thomas}

\address{\small Guram Donadze: \;\rm Indian Institute of Science Education and Research\\}
\email{gdonad@gmail.com}

\address{\small Manuel Ladra: \;\rm University of Santiago de Compostela\\}
\email{manuel.ladra@usc.es}

\address{\small Viji Z. Thomas: \;\rm Indian Institute of Science Education and Research\\}
\email{vthomas@iisertvm.ac.in}

\begin{abstract}
We prove that the class of nilpotent by finite, solvable by finite, polycyclic by finite, nilpotent
of nilpotency class $n$ and supersolvable groups are closed under the formation of the non-abelian tensor product.
We provide necessary and sufficient conditions for the non-abelian tensor product of finitely generated groups to be
finitely generated. We prove that central extensions of most finite simple groups have trivial Bogomolov multiplier.
\end{abstract}


\maketitle

\section{Introduction}

One of the objectives of this paper is to study some closure and finiteness properties of the non-abelian tensor product $G\otimes H$ of groups.
 R. Brown and J.-L. Loday introduced the non-abelian tensor product $G\otimes H$ for a pair of groups $G$ and $H$ in \cite{BroLod84} and \cite{BroLod87}
  in the context of an application in homotopy theory, extending the ideas of J.H.C. Whitehead in \cite{Whi52}.
   We were naturally led to the study of the closure properties of non-abelian tensor product of groups while considering the question
    whether the Schur multiplier of Noetherian groups is finitely generated. Our other objective is to study the Bogomolov multiplier.
     The authors of \cite{KaKu} study groups for which the Bogomolov multiplier is trivial. We prove the triviality of the Bogomolov multiplier
      for some class of groups. The Bogomolov multiplier can be seen as an obstruction to Noether's rationality problem.
       In the last few years, there has been a lot of research on the class of groups with trivial and non-trivial Bogomolov multiplier
       (see \cite{BMP,Ku10,Mo12,Ka14,KaKu}). Except for the Chevalley and Steinberg groups,
       the Schur multiplier of most of the other finite simple groups have order at most 2. There are a very few exceptions to this.
        Keeping this in mind, we show that central extensions of most of the finite simple groups have trivial Bogomolov multiplier.

In \cite{Ell87f} and \cite{Tho10}, the authors prove that the non-abelian tensor product of finite groups is a finite group,
 and they also show that the non-abelian tensor product of finite $p$-groups is a finite $p$-group. In \cite{Vi99},
  Visscher proved that if $G$, $H$ are solvable (nilpotent), then $G\otimes H$ is solvable (nilpotent).
   In \cite{Na00}, Nakaoka also proved that if $G$ and $H$ are solvable, then $G\otimes H$ is solvable.
    She obtains the derived and lower central series of the non-abelian tensor product of groups.
     The author of \cite{Vi99} gives a bound on the nilpotency class of $G\otimes H$ in terms of the derivative subgroup $D_H(G)\trianglelefteq G$.
      We prove that the non-abelian tensor product of groups of nilpotency class at most $n$ is a group of nilpotency class at most $n$,
       thereby improving the bound given by Visscher in \cite{Vi99}. As a corollary, we obtain a bound on the nilpotency class of $G\otimes G$
       which is an improvement of the bound obtained by the authors of \cite{BKM}.
       In \cite{Mo07}, Moravec proved that if $G$ and $H$ are polycyclic groups, then $G\otimes H$ is a polycyclic group.
        So the study of such closure properties has been a recurring theme in the study of non-abelian tensor product of groups.

       We will briefly describe the organization of the paper.
       In Section~\ref{S:prel}, we list some known results to make the exposition self contained and also because we use those results extensively throughout the paper.
         In Section~\ref{S:clos}, we give short proofs of the main results in \cite{Vi99}. We also prove that if $G$ and $H$ are supersolvable groups,
          then $G\otimes H$ is a supersolvable group. Recently the authors of \cite{BaRo} prove that the non-abelian tensor square
           of nilpotent by finite group is a nilpotent by finite group. We prove that the non-abelian tensor product
           of nilpotent by finite groups is a nilpotent by finite group. We also prove that the non-abelian tensor product
            of solvable by finite groups is solvable by finite. Furthermore, we prove that the non-abelian tensor product
             of locally finite, locally solvable, locally nilpotent, locally polycyclic and locally supersolvable has
             the same property respectively. We prove all of the above results using a general strategy and thereby
              bringing all of the above closure properties under one umbrella.

 In Section~\ref{S:more}, we prove the finiteness of $G\otimes H$ under a more general set up. For this, we consider the class of groups $G$
  which is an extension of a finitely generated non-abelian free group by a finite group or it is an extension of a finite group
   by a finitely generated non-abelian free group. With this setup, we prove that if $H$ is a finite group, then $G\otimes H$
   is a finite group. As a consequence, we prove that if $G$ is a finitely generated group and $H$ is a finite group which act
   on each other compatibly, with the action of $H$ on $G$ being trivial, then $G\otimes H$ is a finite group.

In Section~\ref{S:tensor}, we address the following question: is the non-abelian tensor product of finitely generated groups finitely generated?
In general this need not be the case. We provide necessary and sufficient conditions for $G\otimes H$ to be finitely generated.
 If $G$ is a Noetherian group, then we give necessary and sufficient conditions for $G\otimes G$ to be a Noetherian group.
  We show that if $G$ and $H$ are polycyclic by finite groups, then $G\otimes H$ is a polycyclic by finite group, and hence a Noetherian group.

In Section~\ref{S:bogo}, we study the behaviour of Bogomolov multiplier under extensions and as a consequence, we show that the Bogomolov multiplier
 of simple by cyclic groups is trivial. We also prove that the central extensions of groups with Schur multiplier of order at most 2
 have trivial Bogomolov multiplier. As a consequence of our results, we also obtain that the Bogomolov multiplier of the non-abelian
  tensor product of finite groups, $B_0(G\otimes H)$ is trivial, provided $G$ is metacyclic, symmetric group $S_n$ ($n\geq 8$),
   or simple groups with Schur multiplier of order at most 2.

\section{Preliminaries} \label{S:prel}

The non-abelian tensor product of groups is defined for a pair of groups that act on each other provided the actions satisfy the
 compatibility conditions of Definition~\ref{D:2.1} below. Note that we write conjugation on the left, so $^gg'=gg'g^{-1}$ for $g,g'\in G$
  and $^gg'\cdot g'^{-1}=[g,g']$ for the commutator of $g$ and $g'$.

\begin{definition}\label{D:2.1}
Let $G$ and $H$ be groups that act on themselves by conjugation and each of which acts on the other. The mutual actions are said to be \emph{compatible} if
\[
  ^{^h g}h'=\; ^{hgh^{-1}}h' \quad \text{ and }  \quad ^{^g h}g'=\ ^{ghg^{-1}}g', \ \text{for all} \ g,g'\in G, h,h'\in H.
\]
\end{definition}

\begin{definition}
If $G$ and $H$ are groups that act compatibly on each other, then the \emph{non-abelian tensor product} $G\otimes H$
is the group generated by the symbols $g\otimes h$ for $g\in G$ and $h\in H$ with relations
\begin{align*}
gg'\otimes h & =(^gg'\otimes \,^gh)(g\otimes h), \\
g\otimes hh' & =(g\otimes h)(^hg\otimes \,^hh'),
\end{align*}
for all $g,g'\in G$ and $h,h'\in H$.
\end{definition}

The special case where $G=H$, and all actions are given by conjugation, is called the \emph{tensor square} $G\otimes G$. The tensor square of a group is always defined.

There exists a homomorphism $\kappa \colon G\otimes G \rightarrow [G,G]$ sending $g\otimes h$ to $[g,h]$.
 Set $J(G)=\Ker(\kappa)$. Its topological interest is the formula $J(G)\cong \pi_{3}(SK(G,1))$, where $SK(G,1)$ is the suspension of $K(G,1)$.
  The group $J(G)$ lies in the centre of $G\otimes G$.

\begin{definition}
A subgroup of $G$ called the  \emph{derivative} of $G$ by $H$ was introduced in \cite{Vi99}. It is defined as $D_H(G)=\left\langle g \ ^hg^{-1}\mid g\in G,h\in H\right\rangle$.
\end{definition}

The following well-known concept of a crossed module can be found in \cite{BroLod87}. In \cite{Weibel},
it appears in relation with the third cohomology group.

\begin{definition} Let $A$ and $B$ be groups. A \emph{crossed module} is a group
homomorphism $\phi \colon A\rightarrow B$ together with an action of $B$ on $A$
satisfying
\[\phi(^ba)=b\phi(a)b^{-1} \qquad \text{and} \qquad  ^{\phi(a)}a'=aa'a^{-1}\,,\] for
all $b\in B$
and $a,a'\in A$.
\end{definition}

The following proposition appears  in \cite[Proposition 2.3]{BroLod87}. We record it here for the ease of access for the reader.

\begin{proposition} Let $\phi \colon G\otimes H \rightarrow D_H(G)$ be
defined by $\phi(g\otimes h)=g \; ^h g^{-1}$. Then the following hold:
\begin{itemize}
  \item[(i)] $\phi$ is a homomorphism;
 \item[(ii)] there is an action of $G$ on $G\otimes H$ defined by
$^x(g\otimes h)=\;^x g\otimes\,^x h$, where $x\in G$;
 \item[(iii)] $\phi \colon G\otimes H \rightarrow D_H(G)$ is a crossed module.
\end{itemize}
\end{proposition}

The following lemma is well known.

\begin{lemma}
 The kernel of a crossed module $\phi \colon A\rightarrow B$ is a central subgroup and the image of $\phi$ is a normal subgroup of $B$.
\end{lemma}

\section{Closure properties of the non-abelian tensor product of groups}\label{S:clos}

If $G$ and $H$ belong to class $\mathfrak{X}$, then does $G\otimes H$ belong to class $\mathfrak{X}$?
This question has been considered by many authors. The authors of \cite{Ell87f}, \cite{Vi99}, \cite{Na00} and \cite{Mo07}
 have considered this question when $\mathfrak{X}$ is the class of finite groups, $p$-groups, solvable groups,
  nilpotent groups and polycyclic groups. The class of supersolvable groups falls between the class of solvable groups
   and the class of finitely generated nilpotent groups. So it is natural to ask if the non-abelian tensor product of supersolvable
    groups is supersolvable. One of our aims in this section is to prove this result. We also prove that the class of nilpotent by finite,
     solvable by finite, locally finite, locally nilpotent, locally solvable, locally polycyclic and locally supersolvable groups
     are closed under the formation of the non-abelian tensor product of groups. Now we will describe the strategy of the proof
      of our main theorem of this section. We consider the central extension $1\to A\to G\otimes H\xrightarrow{\;\phi\;} D_H(G)\to 1$.
       Since $\phi$ is a crossed module, $A$ is a central subgroup. If $D_H(G)$ is solvable or a nilpotent group, then $G\otimes H$
        being a central extension of $D_H(G)$ inherits that property as well.  Thus we obtain the main result of \cite{Vi99}
         as an immediate corollary of our strategy described above.

\begin{corollary}\label{strategy}
Let $G$ and $H$ be groups acting on each other and acting on themselves by conjugation. If the mutual actions are  compatible, then the following hold:

\item[(i)] If $D_H(G)$ is abelian, then $G\otimes H$ is metabelian.
\item[(ii)] If $D_H(G)$ is solvable, then $G\otimes H$ is solvable.
\item[(iii)] If $D_H(G)$ is nilpotent, then $G\otimes H$ is nilpotent.
\end{corollary}

The authors of \cite{Vi99} and \cite{Na00} show that if $G$ and $H$ are nilpotent groups of class $n$, then $\cl(G\otimes H)\leq \cl(D_H(G))+1$,
 which can also be seen from our strategy described above. It may happen that the nilpotency class of $D_H(G)$ is $n$,
  in which case the above formula gives an upper bound of  $n+1$.  In the next proposition, we improve this bound and
  it provides another example of the closure property of the non-abelian tensor product of groups.

\begin{proposition}\label{nilpotent}
Let $G$ and $H$ be nilpotent groups of nilpotency class $n$ acting on each other. If the mutual actions are compatible, then $G\otimes H$ is a nilpotent group of class at most $n$.
\end{proposition}
\begin{proof} We will show that $(n+1)$-th term of the lower central series $\gamma_{n+1}(G\otimes H)$ is trivial. For this
we show that $xyx^{-1}=y$ for each $x\in \gamma_n(G\otimes H)$ and $y\in G\otimes H$. It suffices to show that conjugating
$g\otimes h$ by $[\dots [[g_1\otimes h_1, g_2\otimes h_2], g_3\otimes h_3], \dots, g_n\otimes h_n]$ fixes $g\otimes h$ for
each $g, g_1, \dots, g_n\in G$ and $h, h_1, \dots, h_n\in H$. By \cite[Proposition 3]{BJR87},
\[
^{(a\otimes b)}(a_1\otimes b_1):=(a\otimes b)(a_1\otimes b_1)(a\otimes b)^{-1}=\;^{[a,b]}(a_1\otimes b_1)=\;^{[a,b]}a_1\otimes\, ^{[a,b]}b_1.
\]
This shows that the action of conjugation by an element $a\otimes b$ is the same as action by $[a,b]$.
Using this and the compatibility of the actions, we obtain
\begin{align*}
&^{[\dots [[g_1\otimes h_1, g_2\otimes h_2], g_3\otimes h_3], \dots, g_n\otimes h_n]}(g\otimes h)
=\;^{[\dots [[[g_1, h_1], [g_2, h_2]], [g_3, h_3]], \dots, [g_n, h_n]]}(g\otimes h)\\
& \ =\;^{[\dots [[[g_1, h_1], [g_2, h_2]], [g_3, h_3]], \dots, [g_n, h_n]]}g \otimes \, ^{[\dots [[[g_1, h_1], [g_2, h_2]], [g_3, h_3]], \dots, [g_n, h_n]]}h\\
& \ =\;^{[\dots [[g_1^{h_1}g_1^{-1}, g_2^{h_2}g_2^{-1}], g_3^{h_3}g_3^{-1}], \dots, g_n^{h_n}g_n^{-1}]}g \otimes \,
^{[\dots [[^{g_1}h_1h_1^{-1}, ^{g_2}h_2h_2^{-1}], ^{g_3}h_3h_3^{-1}], \dots, \:^{g_n}h_nh_n^{-1}]}h \\
& \ = g\otimes h.  \qedhere
\end{align*}
\end{proof}

If $G$ is a nilpotent group of nilpotency class $n$, then by \cite[Proposition 2.2]{BKM} $\cl(G\otimes G)=\cl([G,G])$ or $\cl([G,G])+1$.
 If the nilpotency class of $G$ is $n$, then clearly the nilpotency class of $[G,G]$ is at most $\frac{n}{2}$.
  Thus using the bound found in \cite{BKM}, we obtain that the $\cl(G\otimes G)=\frac{n}{2}+1$.
  The next corollary gives an improvement of this bound and the bounds given in \cite{BJR87}, \cite{Vi99} and \cite{Na00}.
   Since the proof is similar to the proof of the previous proposition, we just record it here without proof.

\begin{corollary}
Let $G$ be a group of nilpotency class $n$. Then the nilpotency class of $G\otimes G$ is bounded above by $\lceil \frac{n}{2} \rceil$, where
 $\lceil x \rceil$ denotes the ceiling function.
\end{corollary}

We do not know whether the analogue of Proposition~\ref{nilpotent} is true for solvable groups. Hence we pose this as a question below.

\

{\bf QUESTION.} Let $G$ and $H$ be solvable groups of solvability length $n$ acting on each other compatibly. Is $G\otimes H$ a solvable group of length at most $n$?

\;

We do not know the answer to the above question even for the case $n=2$, i.e. whether the tensor product of metabelian groups is a metabelian group.
 By Corollary~\ref{strategy}(iii), we obtain the following result, which can also be obtained by the results in \cite{Vi99}.

\begin{corollary}
Let $G$ be a group. If $G$ is metabelian, then $G\otimes G$ is metabelian.
\end{corollary}

In the next theorem, we prove that the property of being supersolvable is closed under formation of non-abelian tensor product of groups.
\begin{theorem}
Let $G$ and $H$ be groups acting on each other compatibly. If $G$ and $H$ are supersolvable, then $G\otimes H$ is supersolvable.
\end{theorem}
\begin{proof} Consider the following exact sequence:
\[
1\to A \to G\otimes H \xrightarrow{\phi_G} D_H(G) \to 1,
\]
where $\phi_G \colon g\otimes h \mapsto  g\;^{h}g^{-1}$ for each $g\in G$, $h\in H$. Since $G\otimes H \to D_H(G)$ is a crossed module, $A$ is a
subgroup of the center of $G\otimes H$. Noting that every supersolvable group is polycyclic, we conclude that $G\otimes H$ is
polycyclic \cite{Mo07}. Hence $A$ is finitely generated and is isomorphic to the direct product of finitely many cyclic
groups,  $A= \displaystyle\mathop{\oplus}\limits_{i=1}^n A_i$. Since $A_i$ is a central subgroup of $G\otimes H$, it is a normal subgroup of $G\otimes H$ for each $1\leq i\leq n$.
The following is an extension of a cyclic group by a supersolvable group, $1\to A_1 \to (G\otimes H)/\displaystyle\mathop{\oplus}\limits_{i=2}^n A_i \to (G\otimes H)/A = D_H(G) \to 1$.
Therefore $(G\otimes H)/\displaystyle\mathop{\oplus}\limits_{i=2}^n A_i$ is supersolvable. Now consider the extension of
groups, $1\to A_2 \to (G\otimes H)/\displaystyle\mathop{\oplus}\limits_{i=3}^n A_i \to (G\otimes H)/\displaystyle\mathop{\oplus}\limits_{i=2}^n A_i \to 1$.
This is also an extension of a cyclic group by supersolvable group implying supersolvability of $(G\otimes H)/\displaystyle\mathop{\oplus}\limits_{i=3}^n A_i$.
Proceeding by induction, we will obtain that $G\otimes H$ is a supersolvable group.
\end{proof}

Using the same strategy as above, we want to examine whether $G\otimes H$ belongs to the class $P$ if either $G$ or $H$ belongs to the class $P$. We begin with the following lemma.

\begin{lemma}
Let $G$ and $H$ be groups acting on each other compatibly. Suppose $P$ is a property of groups that satisfies the following conditions:
\item[(i)] $P$ is closed under taking normal subgroups;
\item[(ii)] If a group has property $P$, then any central extension of that group has property $P$.

Then $G\otimes H$ has property $P$ as long as one of $G$ or $H$ has property $P$.
\end{lemma}
\begin{proof}
The lemma follows easily by considering the central extensions $1\to \Ker \phi_G \to G\otimes H \xrightarrow{\phi_G} D_H(G)\to 1$
or $1\to  \Ker \phi_H \to G\otimes H \xrightarrow{\phi_H} D_G(H)\to 1$.
\end{proof}

\begin{corollary}
Let $G$ and $H$ be groups acting on each other and acting on themselves by conjugation. If the mutual actions are  compatible, then the following hold:
\item[(i)] If $G$ or $H$ is solvable by finite, then $G\otimes H$ is solvable by finite;
\item[(iii)] If $G$ or $H$ is nilpotent by finite, then $G\otimes H$ is nilpotent by finite.
\end{corollary}
\begin{proof}
By the previous lemma, it suffices to prove that the property solvable by finite or nilpotent by finite is closed under taking normal subgroups
and also closed under taking central extensions. We will prove the result assuming $G$ is solvable by finite, the other case follows similarly.
 So we have an exact sequence $1\to S\to G\to F\to 1$, where $S$ is a solvable subgroup of $G$ and $F$ is a finite group.
  Suppose $N$ is a normal subgroup of $G$, our aim is to show that $N$ is solvable by finite. Consider the exact sequence,
   $1\to N\cap S\to N\to F_1\to 1$. Clearly $F_1$ is a finite group and $N\cap S$ is a subgroup of $S$ and hence solvable.
    Thus $N$ is a solvable by finite group. Now consider the central extension $1\to C\to E\xrightarrow{\, f \,} G\to 1$.
     Our aim is to show that $E$ is a solvable by finite group. To see this, first consider the central extension
      $1\to C\to f^{-1}(S)\xrightarrow{\, f \,} S\to 1$. Note that $f^{-1}(S)$ is a solvable group as it is a central extension of a solvable group.
       Finally consider the following extension of groups, $1\to f^{-1}(S)\to E\to F\to 1$ to obtain the desired result.
\end{proof}

\begin{remark}
Note that in the proof of the previous result, we only require $D_G(H)$ or $D_H(G)$ to have property $P$.
\end{remark}

We have already seen that if $G$ and $H$ belong to the class of finite, solvable, supersolvable, nilpotent or polycyclic groups,
 then $G\otimes H$ also belongs to the same class, respectively. It is natural to ask if the same is true if we replace the property $P$
  by the property locally $P$. We say that the property $P$ is closed under forming the non-abelian tensor product of groups if $G$ and $H$
  have property $P$ implies that $G\otimes H$ has property $P$. With this terminology, we state the following lemma.

\begin{lemma}
Let $G$ and $H$ be groups acting on each other compatibly. Suppose $P$ is a property of groups that satisfies the following conditions:
\item[(i)] $P$ is closed under taking subgroups and homomorphic images;
\item[(ii)] $P$ is closed under forming the non-abelian tensor product of groups.

Then $G\otimes H$ is locally $P$ provided $G$ and $H$ are locally $P$.
\end{lemma}
\begin{proof}
We need to show that any finitely generated subgroup of $G\otimes H$ has property $P$. Let $X$ be a finitely generated subgroup of $G\otimes H$.
 Suppose it is generated by $x_1,\dots, x_t$, where each $x_i=\prod_j g_{i_j}\otimes h_{i_j}$. Let $G_1$ be the subgroup of $G$
  generated by $g_{i_j}$ for all $i$ and all $j$. Let $H_1$ be the subgroup of $H$ generated by $h_{i_j}$ for all $i$ and all $j$.
   By assumption $G_1$ and $H_1$ have property $P$ and hence $G_1\otimes H_1$ has property $P$. Consider the natural homomorphism
    from $\phi \colon G_1\otimes H_1\to G\otimes H$. Clearly $X$ is a subgroup of the image of $\phi$ and hence has property $P$.
\end{proof}

As an immediate corollary, we obtain the following result.

\begin{corollary}
Let $G$ and $H$ be groups acting on each other and acting on themselves by conjugation. If the mutual actions are  compatible, then the following hold:
\item[(i)] If $G$ and $H$ are locally finite, then $G\otimes H$ is locally finite;
\item[(ii)] If $G$ and $H$ are locally solvable, then $G\otimes H$ is locally solvable;
\item[(iii)] If $G$ and $H$ are locally nilpotent, then $G\otimes H$ is locally nilpotent;
\item[(iv)] If $G$ and $H$ are locally polycyclic, then $G\otimes H$ is locally polycyclic;
\item[(v)] If $G$ and $H$ are locally supersolvable, then $G\otimes H$ is locally supersolvable.
\end{corollary}

\section{More on finiteness of $G\otimes H$}\label{S:more}

The finiteness of $G\otimes H$ when $G$ and $H$ are finite has been the topic of \cite{Ell87f} and \cite{Tho10}.
 In this section we will show that $G\otimes H$ is finite in a more general set up. If $G$ is finitely generated and $H$ is finite,
  then $G\otimes H$ need not be finite as the following example shows.

\begin{example}
Let $H$ be a finite group and $G=\mathbb{Z}(H)$ the underlying abelian group of the integral group ring of $H$.
 Define an action of $H$ on $\mathbb{Z}(H)$ via the multiplication in $\mathbb{Z}(H)$.
Moreover, suppose that $\mathbb{Z}(H)$ acts trivially on $H$.
Then, we have mutual compatible actions of $G$ and $H$, and by \cite{Gu88} there is an isomorphism
$H\otimes \mathbb{Z}(H)=I(H)\otimes_H \mathbb{Z}(H)=I(H)$. But $I(H)$ is not finite for $H\neq \{1\}$.
\end{example}

In this section we will show that if $G$ is finitely generated and $H$ is finite, then $G\otimes H$ is finite provided $G$ is from the class defined below.

\subsection{Definition of a class $\mathcal{C}$}
We say that a group $G$ belongs to a class $\mathcal{C}$ and write $G\in \mathcal{C}$, if either $G$ is an extension of a finite group by a finitely generated
non-abelian free group, or $G$ is an extension of a finitely generated non-abelian free group by a finite group,
i.e. we have one of the following extensions of groups:
\begin{align}
& 1\to Q\to G \to F \to 1, \label{C1} \tag{$\mathcal{C}_1$}\\
& 1\to F\to G \to Q \to 1, \label{C2} \tag{$\mathcal{C}_2$}
\end{align}
where $Q$ is a finite group and $F$ is a finitely generated non-abelian free group.

\begin{lemma} \label{lemma1}
Let $G$ be a group from the class $\mathcal{C}$.
Then $H_n(G)$ is finite for all $n\geq 2$, where $H_n(G)$ denotes the $n$-th Eilenberg-MacLane homology group.
\end{lemma}
\begin{proof} Case~\ref{C1}: Suppose that $G$ is an extension of a finite group $Q$ by a finitely generated
non-abelian free group $F$:
\[
1\to Q\to G \to F \to 1.
\]
We have the Hochschild-Serre spectral sequence:
\[
H_p(F, H_q(Q))\Rightarrow H_{p+q} (G).
\]
Since $H_p(F, H_q(Q))=0$ for all $p\geq 2$, it suffices to show that
$H_0(F, H_q(Q))$ and $H_1(F, H_q(Q))$ are finite for all $q\geq 1$. Note that $H_q(Q)$ is finite for all $q\geq 1$.
Since $F$ is finitely generated, its homology groups with coefficients in finite $F$-modules are finite.

Case~\ref{C2}: Suppose that $G$ is an extension of a finitely generated non-abelian free group $F$ by a finite group $Q$.
We have the Hochschild-Serre spectral sequence:
\[
H_p(Q, H_q(F))\Rightarrow H_{p+q} (G).
\]
Since $H_q(F)=0$ for all $q\geq 2$, it suffices to show that $H_p(Q, H_0(F))$ and $H_p(Q, H_1(F))$
are finite for all $p\geq 1$. Note that both $H_0(F)$ and $H_1(F)$ are finitely generated abelian groups.
Since $Q$ is finite, its homology groups in positive dimensions with coefficients in finitely generated $Q$-modules are finite.
\end{proof}

\begin{lemma}\label{lemma2}
Let $G$ be a group from the class $\mathcal{C}$ and $N$ be an abelian normal subgroup of $G$. Then $N$ is finite and $G/N\in \mathcal{C}$.
\end{lemma}
\begin{proof} Case~\ref{C1}: Suppose that $G$ is an extension of a finite group $Q$ by a finitely generated
non-abelian free group $F$. Then $NQ/Q$ is a normal subgroup of $G/Q=F$. Since $F$ is a non-abelian free group, $F$ does not contain
a nontrivial abelian normal subgroup. Thus, $NQ/Q =\{1\} \Rightarrow N\subseteq Q \Rightarrow N$ is finite. Moreover, we have an extension
\[
1 \to Q/N \rightarrow G/N \rightarrow F \rightarrow 1
\]
implying that $G/N \in \mathcal{C}$.

Case~\ref{C2}: Suppose that $G$ is an extension of a finitely generated non-abelian free group $F$ by a finite group $Q$. Then, $N\cap F=\{1\}$
because $N\cap F$ is an abelian normal subgroup of a non-abelian free group $F$. Hence, $N=NF/F \subseteq G/F=Q$. This implies that $N$ is
finite. Moreover, we have the following extension:
\[
1 \to F \to G/N \to G/ (NF)\to 1.
\]
Since $G/ (NF)$ is a quotient of a finite group $Q$, $G/N \in \mathcal{C}$.
\end{proof}

\begin{lemma}\label{lemma3}
Let $G$ be a group from the class $\mathcal{C}$ and $N$ be a finite normal subgroup of $G$. Then $G/N\in \mathcal{C}$.
\end{lemma}
\begin{proof} Note that any free group does not contain a finite nontrivial subgroup. The rest of the proof follows Lemma~\ref{lemma2}
mutatis mutandis.
\end{proof}

\begin{lemma}\label{lemma4}
Let $G$ and $H$ be normal subgroups of some group. Suppose that $H$ is finite and $G\in \mathcal{C}$. Then $H_n(GH)$ is finite
for all $n\geq 2$.
\end{lemma}
\begin{proof}
Denote the quotient group $GH/G$ by $H'$ and consider the Hochschild-Serre spectral sequence:
\[
H_p(H', H_q(G)) \Rightarrow H_{p+q}(GH).
\]
By Lemma~\ref{lemma1} we have that $H_q(G)$ is finite for all $q\geq 2$. Since $H'$ is finite, we have
that $H_p(H', H_q(G))$ is finite for all $q\geq 2$. Moreover, both $H_0(G)$ and $H_1(G)$ are finitely generated abelian groups. This implies that
$H_p(H', H_0(G))$ and $H_p(H', H_1(G))$ are finite for all $p\geq 1$. Hence $H_p(H', H_q(G))$ is finite for all $p+q\geq 2$.
\end{proof}

The idea of the proof of the next result is inspired by the proof of the main result in \cite{Ell87f}.

\begin{theorem} \label{finite}
Let $G$ be a group belonging to the class $\mathcal{C}$ and let $H$ be a finite group. If $G$ and $H$ act on each other compatibly, then $G\otimes H$ is finite.
\end{theorem}
\begin{proof}
Special Case: Suppose that $G$ and $H$ are normal subgroups of some group and that they act on each other by conjugation. From \cite{BroLod84} we have an exact
sequence
\begin{equation}\label{E:h3}
 H_3(GH/H)\oplus H_3(GH/G) \to (\Ker [\;,\;] \colon G\wedge H \to [G, H])\to H_2(GH),
\end{equation}
where $[\;,\;] \colon G\wedge H \to [G, H]$ is defined by $g\wedge h \mapsto [g, h]$ for all $g\in G$ and $h\in H$. Since $GH/G$ is finite,
$H_3(GH/G)$ is also finite. Since $GH/H$ is the quotient of $G$ by a finite group, Lemmas~\ref{lemma3} and  \ref{lemma1} imply that
$H_3(GH/H)$ is finite. By Lemma~\ref{lemma4} we have that $H_2(GH)$ is also finite. Hence $G\wedge H$ is finite.
From \cite{BroLod87} we have an exact sequence
\begin{equation}\label{E:gamma}
 \Gamma (G\cap H /[G, H])\to G\otimes H \to G\wedge H \to 1,
\end{equation}
where $\Gamma$ is the  Whitehead's universal quadratic functor \cite{Whi52}. Since $G\cap H /[G, H]$ is finite abelian group,
$\Gamma (G\cap H /[G, H])$ is finite. Hence $G\otimes H$ is finite.

General Case: Suppose that $G$ and $H$ are as in the proposition. Let $(G, H)$ be the normal subgroup of the semidirect product $G\rtimes H$
generated by the elements $(g\,^hg^{-1}, h\,^gh^{-1})$  for all $g\in G$ and $h\in H$. Set $G \circ H = G \rtimes H / (G, H)$. There is an action
of $G \circ H$ on $G$ and on $H$ given by $^{(g,h)}g'=\;^g(^hg')$ and $^{(g,h)}h'=\;^g(^hh')$ for $g, g'\in G$, $h, h'\in H$, and the
natural homomorphisms $\mu \colon G \to G \circ H$ and $\nu \colon H \to G \circ H$ together with these actions are crossed modules. Hence, $\Ker \mu$
and $\Ker\nu$ are abelian groups acting trivially on $H$ and on $G$, respectively. Therefore, by \cite{Gu88} we have that
$G\otimes \Ker \nu  = I(G)\otimes _{\mathbb{Z}(G)}\Ker\nu $, which is finite because $\Ker\nu$ is finite and $I(G)$ is finitely generated
$\mathbb{Z}(G)$-module. By Lemma~\ref{lemma2} we have that $\Ker\mu$ is finite. This implies that $H\otimes \Ker \mu$ is finite.
Since $\mu \colon G \to G \circ H$ and $\nu \colon H \to G \circ H$ are crossed modules, $\mu (G)$ and $\nu (H)$ will be normal subgroups of $G \circ H$.
Moreover, $\mu (G)\in \mathcal{C}$ (by Lemma~\ref{lemma2}) and $\nu (H)$ is finite. Hence $\mu (G)\otimes \nu (H)$ is finite. Thus, the
exact sequence
\begin{equation}\label{E:ker}
 G\otimes \Ker \nu \oplus H\otimes \Ker \mu \to G\otimes H \to \mu (G)\otimes \nu (H) \to 1
\end{equation}
implies that $G\otimes H$ is finite.
\end{proof}

Straightforward from this proposition we have the following result.

\begin{corollary}
Let $F$ be a finitely generated non-abelian free group and let $H$ be a finite group. If $F$ and $H$ act on each other compatibly,
then $F\otimes H$ is finite.
\end{corollary}

\begin{corollary}
Let $G$ be a finitely generated group and $H$ be finite. Suppose that $G$ and $H$ act on each other compatibly and that $H$ acts on $G$
trivially. Then $G\otimes H$ is finite.
\end{corollary}
\begin{proof}
There exist a finitely generated non-abelian free group $F$ and an epimorphism $\tau \colon F\to G$. Define
an action of $F$ on $H$ by $^xh=\;^{\tau(x)}h$ for all $x\in F$ and $h\in H$. Moreover, suppose that $H$ acts trivially on $F$. Then we have
mutual compatible actions of $F$ and $H$, and we have an epimorphism $F\otimes H \to G\otimes H$ induced by $\tau$.
By the previous corollary $F\otimes H$ is finite. Hence $G\otimes H$ is finite.
\end{proof}

Let $G$ be an extension of a finite group by a finitely generated free abelian group or an extension of a finitely generated free abelian group by a finite group.
 Then $G$ does not belong to $\mathcal{C}$. If $H$ is a finite group, then the next remark shows that $G\otimes H$ need not be finite.

\begin{remark}
Suppose that $G$ is either an extension of $\mathbb{Z}$ by a finite group, or an extension of a finite group
by $\mathbb{Z}$. In this case $G\otimes H$ is not always finite for a finite group $H$. For instance, assume
that $G=\mathbb{Z}$ and $H=\mathbb{Z}/2\mathbb{Z}=\langle t \mid t+t=0\rangle$. Define an action of $H$ on $\mathbb{Z}$
by $^tn=-n$ for each $n\in \mathbb{Z}$. Moreover, assume that $\mathbb{Z}$ acts trivially on $H$. Then $\mathbb{Z}$ and $H$
act on each other compatibly but $\mathbb{Z}\otimes H$ is isomorphic to $\mathbb{Z}$.
\end{remark}

\section{Tensor product of finitely generated groups} \label{S:tensor}

In \cite{BCM}, the authors prove that the integral homology and cohomology groups of polycyclic by finite groups is finitely generated.
 So it is natural to ask if the same result holds for Noetherian groups. In this paper, we want to restrict our attention
  to the study of Schur multiplier of Noetherian groups.  A very natural approach to study this problem is to consider a Noetherian group $G$
   and look at its tensor square $G\otimes G$. If we can prove that the tensor square is a Noetherian group,
   then it follows that the exterior square is Noetherian, and  thereby the Schur multiplier is Noetherian.
    It is more natural to consider the class of finitely generated groups $G$ and $H$ and to study the properties of the non-abelian tensor product
     of finitely generated groups before we embark on the study of tensor squares of Noetherian groups.
      First notice that for finitely generated groups $G$ and $H$ acting on each other compatibly,
       their non-abelian tensor product need not always be finitely generated. For example, if $G$ is a non-abelian free group,
        then its tensor square is not finitely generated. But we have the following necessary and sufficient conditions for $G\otimes H$ to be finitely generated.

\begin{proposition}
Let $G$ and $H$ be groups acting on each other compatibly. If $G$ and $H$ are finitely generated, then $G\otimes H$ is finitely generated if and only if
$D_G(H)$ and $D_H(G)$ are finitely generated.
\end{proposition}
\begin{proof}
One direction is clear because we have well defined epimorphisms $G\otimes H \to D_G(H)$ and $G\otimes H \to D_H(G)$.

We aim to show that if $D_G(H)$ and $D_H(G)$ are finitely generated, then so is $G\otimes H$.
Suppose that $x_1, \dots, x_n \in G$ are generators for $G$ and that $y_1, \dots, y_m \in H$ are generators for $H$.
Suppose $D_G(H)$ has generators of the form $^{g_1}h_1h_1^{-1}, \dots, \;^{g_p}h_ph_p^{-1}$,
for $g_1, \dots, g_p\in G$ and $h_1, \dots, h_p\in H$. Similarly $D_H(G)$ has generators of the form $(g'_1)\:^{h'_1}(g'_1)^{-1}, \dots, \,
(g'_q)\:^{h'_q}(g'_q)^{-1}$  for $g'_1, \dots, g'_q\in G$ and $h'_1, \dots, h'_q\in H$.
We will show that the following elements
\begin{equation}\label{generators}
x_i^\alpha\otimes y_j^\beta, \; g_i\otimes h_i, \; g'_j\otimes h'_j, \; x_i^\alpha\otimes (^{g_j}h_jh_j^{-1})^\beta, \;
((g'_i)^{h'_i}(g'_i)^{-1})^\alpha \otimes y_j^\beta,
\end{equation}
for $\alpha, \beta \in \{1, -1\}$, generate $G\otimes H$. Using the defining relations of non-abelian tensor product
it is easy to see that each element of $G\otimes H$ can be factored into a product of the elements $^z(x_i^\alpha\otimes y_j^\beta)$
for $z\in G*H$. Hence, it is enough to show that this element can be factored into a product of elements from the list \eqref{generators}
and their inverses. For the latter, it suffices to show that if $a\otimes b \in G\otimes H$ is an element from the list \eqref{generators},
then $^{x_i^{\alpha}}(a\otimes b)$ and $^{y_i^{\alpha}}(a\otimes b)$ can be factored into a product of elements
from the list \eqref{generators} and their inverses. We have the following relations \cite{BJR87}:
\begin{align*}
^{x_i^{\alpha}}(a\otimes b)& =(x_i^{\alpha} \otimes \,^{a}bb^{-1})(a\otimes b); \\
^{y_i^{\alpha}}(a\otimes b) & =(a\otimes b)(a \, ^ba^{-1}\otimes y_i^{\alpha})^{-1}.
\end{align*}
Hence to finish the proof we need to show that $x_i^{\alpha} \otimes \,^{a}bb^{-1}$ and $a \, ^ba^{-1}\otimes y_i^{\alpha}$
can be factored into a product of elements from the list \eqref{generators} and their inverses. We have $^{a}bb^{-1}\in D_G(H)$,
hence it is a finite product of $^{g_j}h_jh_j^{-1}$'s and their inverses. Suppose that $^{a}bb^{-1}=\;^{g_1}h_1h_1^{-1} \overline{h}$
where $\overline{h}$ is also a finite product of $^{g_j}h_jh_j^{-1}$'s and their inverses. Then
\begin{align*}
x_i^{\alpha} \otimes \,^{a}bb^{-1} & = x_i^{\alpha} \otimes \,^{g_1}h_1h_1^{-1} \overline{h} =
(x_i^{\alpha} \otimes \,^{g_1}h_1h_1^{-1})^{^{g_1}h_1h_1^{-1}}(x_i^{\alpha} \otimes \overline{h})\\
{}&=
(x_i^{\alpha} \otimes \,^{g_1}h_1h_1^{-1})^{[g_1, h_1]}(x_i^{\alpha} \otimes \overline{h})\\
{}&=
(x_i^{\alpha} \otimes \,^{g_1}h_1h_1^{-1})(g_1\otimes h_1)(x_i^{\alpha} \otimes \overline{h})(g_1\otimes h_1)^{-1}.
\end{align*}
Proceeding by induction we will obtain that $x_i^{\alpha} \otimes \,^{a}bb^{-1}$ can be written as a product of elements
 from the list \eqref{generators} and their inverses. Similarly we can prove the same for $a \, ^ba^{-1}\otimes y_i^{\alpha}$.
\end{proof}

\begin{corollary}
Let $G$ be a finitely generated group. Then $G\otimes G$ is finitely generated if and only if $[G, G]$ is finitely generated.
\end{corollary}

\begin{corollary}
Let $G$ and $H$ be Noetherian groups acting on each other compatibly. Then $G\otimes H$ is finitely generated.
\end{corollary}

\begin{proposition}
Let $G$ and $H$ be Noetherian groups acting on each other compatibly. If $D_H(G)$ is finite, then $G\otimes H$ is Noetherian.
\end{proposition}
\begin{proof}
Consider the short exact sequence $1\to A\to G\otimes H\to D_H(G)\to 1$, where $A$ is a central subgroup of $G\otimes H$.
 Since $A$ is a finite index subgroup of a finitely generated group, it is a finitely generated abelian group, and hence Noetherian.
  Thus $G\otimes H$ is a Noetherian group, as it is an extension of a Noetherian group by a Noetherian group.
\end{proof}

\begin{corollary}
Let $G$ be a finite group and let $H$ be a Noetherian group acting on each other. If the mutual actions are compatible, then $G\otimes H$ is a Noetherian group.
\end{corollary}

\begin{corollary}
Let $G$ and $H$ be Noetherian groups acting on each other compatibly. If one of them acts on the other trivially, then $G\otimes H$ is a Noetherian group.
\end{corollary}
\begin{proof}
Since one of the group acts on the other trivially, either $D_H(G)$ or $D_G(H)$ is trivial, and hence the result.
\end{proof}

We are interested in the class of groups for which $[G,G]$ is finite. In considering this class, we are naturally led to the following class of groups.

\begin{definition}
A group $G$ is called a BFC-\emph{group} if each conjugacy class is finite and the number of its elements does not exceed some number $d=d(G)$.
\end{definition}

In \cite{Neu54}, B. H. Neumann  has characterized  BFC-groups in the following way: a group $G$ is a BFC-group iff $[G,G]$ is finite.
 With this result in hand and noting that when $G=H$ and the groups are acting on each other by conjugation, then $D_H(G)=[G,G]$,
  we state the following corollary.

\begin{corollary}
Let $G$ be a Noetherian BFC group. Then $G\otimes G$ is Noetherian.
\end{corollary}

We do not know whether $G\otimes G$ is Noetherian for a Noetherian group $G$. In the next proposition, we show that this problem can be reduced to studying
the Schur multiplier of Noetherian groups.

\begin{proposition}
Let $G$ be a Noetherian group. Then $G\otimes G$ is Noetherian if and only if the Schur multiplier $M(G)$ is finitely generated.
\end{proposition}
\begin{proof}
Since $G$ is Noetherian, the following exact sequence
\[
1\to M(G) \to G\wedge G \to [G, G] \to 1
\]
implies that $G\wedge G$ is Noetherian if and only if $M(G)$ is finitely generated. On the other hand, since the Whitehead
quadratic functor $\Gamma (\Gab)$ is finitely generated abelian group and
\[
\Gamma (\Gab) \to G\otimes G \to G\wedge G \to 1
\]
is exact, we get that $G\wedge G$ is Noetherian if and only if $G\otimes G$ is Noetherian.
\end{proof}

\begin{proposition}
Let $G$ be a finitely presented Noetherian group. Then $G\otimes G$ is Noetherian.
\end{proposition}
\begin{proof}
Straightforward from the fact that the Schur multiplier of a finitely presented group is finitely generated.
\end{proof}

 In \cite{Mo07}, Moravec proves that the non-abelian tensor product of polycyclic groups is polycyclic.
 A group is polycyclic iff it is solvable and Noetherian. We have seen that if $G$ is a finite group and $H$ is a Noetherian group,
  then $G\otimes H$ is a Noetherian group. So the following question is very natural.

  \

{\bf QUESTION.} Let $G$ and $H$ be Noetherian groups acting on each other compatibly. Is $G\otimes H$ a Noetherian group?

In general, we do not know the answer to this question. The class of polycyclic by finite groups is a more general class than the class of polycyclic groups.
 A polycyclic by finite group need not be solvable but it enjoys lot of the finiteness properties that a polycyclic group has.
  So it is natural to ask if the non-abelian tensor product of polycyclic by finite groups is polycyclic by finite.
   Using the methods of Theorem~\ref{finite}, we can prove the following theorem, which provides further sufficient conditions for $G\otimes H$
    to be a Noetherian group. Since the method of the proof is similar to that of Theorem~\ref{finite}, we just sketch the proof here.

\begin{theorem}
Let $G$ and $H$ be polycyclic by finite groups. If $G$ and $H$ act on each other compatibly, then $G\otimes H$ is polycyclic by finite. In particular, $G\otimes H$ is a Noetherian group.
\end{theorem}
\begin{proof}
By \cite{BCM}, it follows that the integral homology groups of polycyclic by finite groups is finitely generated.
 Hence the  exact sequence \eqref{E:h3}  implies that $G\wedge H$ is a polycyclic by finite group because
  it is an extension of a polycyclic group by a polycyclic by finite group. Now the exact sequence \eqref{E:gamma}
   implies that $G\otimes H$ is a polycyclic by finite group as it is an extension of a polycyclic group by a polycyclic by finite group.
   For the general case, first note that the image of $G\otimes \Ker \nu \oplus H\otimes \Ker \mu$ in the exact sequence \eqref{E:ker} is a central subgroup.
    Thus $G\otimes H$ is a polycyclic by finite group as it is an extension of a polycyclic group by a polycyclic by finite group.
\end{proof}

\section{More on the Bogomolov multiplier}\label{S:bogo}

The main object of this section is the group  $B_0(G)=\Ker \{H^2(G,\mathbb{Q}/{\mathbb{Z}})\to \displaystyle\bigoplus\limits_{A} H^2(A, \mathbb{Q}/{\mathbb{Z}})\}$,
 where $A$ runs over all abelian subgroups of a finite group $G$. Bogomolov (\cite{Bo88}) showed that this group coincides with the unramified Brauer group $\Brnr(V/G)$,
  where $V$ is a vector space defined over an algebraically closed field $k$ of characteristic zero and equipped with a faithful, linear and generically free action of $G$.
   Saltman (\cite{Sa84}) used this to produce the first counterexample to a  problem of Noether on rationality of fields of invariants $k(x_1,\dots, x_n)^G$,
    where $G$ acts on the variables $x_i$ by permutation. More recently, Moravec \cite{Mo12} has given an alternate description for the Bogomolov multiplier
     and we briefly describe this now. Suppose that $M_0(G)$ is a subgroup of $G\wedge G$ generated by all elements $x\wedge y$ such that
$[x, y]=1$. Each such element is contained in the center of $G\wedge G$. Therefore $M_0(G)$ is a normal subgroup of
$G\wedge G$. It is shown in \cite{Mo12} that the Bogomolov multiplier $B_0(G)$ of $G$ is isomorphic to the following group,
\[
B_0(G)\cong \Hom \Big ( \Ker \{G\wedge G /M_0(G) \to [G, G]\}, \mathbb{Q}/\mathbb{Z}\Big),
\]
where $G\wedge G /M_0(G) \to [G, G]$ is a map induced by $g\wedge g' \mapsto [g, g']$. Hence we have an isomorphism,
$B_0(G)\cong \Hom \big( M(G) /M_0(G), \mathbb{Q}/\mathbb{Z}\big)$, where $M(G)$ denotes the Schur multiplier of $G$.

Given an extension of groups $1 \to N \to G \to H \to 1 $, we have the following sequence $B_0(H) \to B_0(G) \to B_0(N)$ which need not
be exact in general. In the following proposition we give a sufficient condition providing exactness of this sequence.

\begin{proposition}
Let $1 \to N \to G \to H \to 1 $ be an extension of groups with $N$ being perfect. If $M_0(G) \to M_0(H)$ is an epimorphism, then
we have the exact sequence
\[
1 \to B_0(H) \to B_0(G) \to B_0(N).
\]
\end{proposition}
\begin{proof}
We have the following exact sequence:
\[
N\wedge G \to G\wedge G \to H\wedge H \to 1.
\]
Since $N$ is perfect, the following relation $[x, y]\otimes g = \;^{g}(x\otimes y)^{-1} (x\otimes y)$ for all $x, y\in N$ and $g\in G$, implies
that the image of $N\wedge G$ in $G\wedge G$ is the same as that of $N\wedge N$. Therefore, the aforementioned exact sequence together with
the epimorphism $M_0(G)\to M_0(H)$ give the following exact sequence:
\[
N\wedge N/M_0(N) \to G\wedge G/M_0(G) \to H\wedge H/M_0(H) \to 0 .
\]
Now using the following diagram with exact rows
\begin{equation*}
\xymatrix@+20pt{
&\frac{N\wedge N}{M_0(N)}\ \ar@{->}[r]
\ar@{->}[d]
 &\frac{G\wedge G}{M_0(G)}\ar@{->}[r]
\ar@{->}[d]
&\frac{H\wedge H}{M_0(H)}\ar@{->}[r]
\ar@{->}[d]
&1 \\
1\ \ar@{->}[r]
&N=[N, N]\ \ar@{->}[r]
 &[G,G]\ar@{->}[r]
&[H,H]\ar@{->}[r]
&1 ,
}\end{equation*}
we get an exact sequence:
\[
M(N)/M_0(N) \to M(G)/M_0(G) \to M(H)/M_0(H) \to 1.
\]
Applying the exact functor $\Hom (-, \mathbb{Q}/\mathbb{Z})$ to the previous exact sequence gives the desired result.
\end{proof}

\begin{corollary}
Suppose we are given an extension of groups $1\to N\to G\to H\to 1$ where $M(H)=0$ and $N$ is perfect.
Then there is an injective homomorphism $B_0(G)\to B_0(N)$.
\end{corollary}
\begin{proof}
Straightforward from the previous proposition because $M_0(H)=0$ and $B_0(H)=0$.
\end{proof}

\begin{corollary}
Suppose that $G=N\rtimes H$ and $N$ is a perfect group. Then there is an exact sequence:
\[
1 \to B_0(H) \to B_0(G) \to B_0(N).
\]
\end{corollary}
\begin{proof}
Since $H$ has a complement in $G$, it is clear that $M_0(G)\to M_0(H)$ is onto.
\end{proof}

\begin{corollary}
Let $G$ be an extension of a finite simple group by a cyclic group. Then $B_0(G)=0$.
\end{corollary}
\begin{proof}
We have an extension of groups
\[
1\to N\to G \to G/N \to 1,
\]
where $N$ is simple group and $G/N$ is cyclic. If $N$ is abelian, then $G$ is an abelian by cyclic group, and hence by  \cite[Theorem 1.2]{KaKu}, its Bogomolov
multiplier is trivial. If $N$ is not abelian, then by the previous corollary there exists an inclusion
$B_0(G)\to B_0(N)$. But on the other hand, $B_0(N)=0$ (see \cite{BMP, Ku10}).
\end{proof}

\begin{remark}
Notice that using the above result, we easily obtain that $B_0(S_n)=0$ for $n\geq 5$, where $S_n$ denotes the symmetric group on $n$ letters.
 This was obtained in \cite{OnWa}, and also mentioned in \cite{KaKu}.
\end{remark}

\begin{lemma}
Let $G$ be a metacyclic group. Then there exists an element $s\in G$ such that
\[
M(G)=\{x\wedge s \mid x\in N,\ [x, s]=1\}.
\]
\end{lemma}
\begin{proof}
We have an extension of groups
\[
1\to N \to G \to H \to 1,
\]
where both $N$ and $H$ are cyclic groups. Since $H\wedge H=1$, we have an epimorphism $N\wedge G \twoheadrightarrow G\wedge G$.
Let $h\in H$ be a generator of $H$. Choose an element $s\in G$ which maps to $h$. Set $S=\langle s \rangle$ and note that $G=SN$.
Since $N$ is a cyclic group, $x\wedge x'=1$ for each $x, x'\in N$. Therefore each element in $N\wedge G$ can be
written as a product $(x_1\wedge s)(x_2\wedge s)\cdots (x_k\wedge s)$, where $x_1, x_2, \dots, x_k\in N$.
Observe that
\begin{align*}
(xx')\wedge s & = (x'\wedge \;^xs)(x\wedge s)=(x'\wedge s[s^{-1}, x])(x\wedge s) \\
&{}=(x'\wedge s)(^sx'\wedge \;^s[s^{-1}, x])(x\wedge s)=(x'\wedge s)(x\wedge s),
\end{align*}
for each $x, x'\in N$. Combining the last two observations, we can conclude that each element of $N\wedge G$ looks like
$x\wedge s$ for some $x\in N$. Since $N\wedge G \to G\wedge G$ is an epimorphism, we obtain that each element of $G\wedge G$ is of the form $x\wedge s$ for some $x\in N$, whence the lemma.
\end{proof}

\begin{lemma} Let $G$ be a group with $M(G)=0$ or $M(G)=\mathbb{Z}/2\mathbb{Z}$. If $B_0(G)=0$, then there exist an element
$s\in G$ such that
\[
M(G)=\{x\wedge s \mid x\in G, [x, s]=1\}.
\]
\end{lemma}
\begin{proof}
If $M(G)=0$, then there is nothing to prove. Suppose that $M(G)=\mathbb{Z}/2\mathbb{Z}$. Since $M(G)=M_0(G)$, there
are elements $s, x\in G$ such that $x\wedge s$ is non-zero in $G\wedge G$ and $[x, s]=1$. Since $M(G)$ contains just two elements, one of them
has to be $x\wedge s$ and the other $1\wedge s$.
\end{proof}

\begin{proposition}
Suppose that $G$ satisfies one of the following conditions:
\begin{itemize}
  \item[(i)] $G$ is metacyclic group;
  \item[(ii)] $M(G)=0$ or $M(G)=\mathbb{Z}/2\mathbb{Z}$ and $B_0(G)=0$.
\end{itemize}
Then for any central extension $1\to C \to X \to G \to 1$, we have $B_0(X)=0$.
\end{proposition}
\begin{proof}
Let $\omega$ be an element in $M(X)$. By the previous
lemma, the image of $\omega$ in $M(G)$ is of the form $x\wedge s$, where $x, s\in G$ and $[x, s]=1$. Suppose that
$x'\in X$ (resp. $s'\in X$) is an element which maps to $x$ (resp. $s$). Then
$\omega = (x'\wedge s')(c_1\wedge x_1)(c_2\wedge x_2)\cdots (c_k\wedge x_k)$, where $c_1, \dots, c_k\in C$ and
$x_1, \dots, x_k \in X$. Since $[c_i, x_i]=1$ for each $i=1, \dots, k$ and $\omega \in M(X)$, we obtain that $[x', s']=1$.
Thus $\omega \in M_0(X)$.
\end{proof}

\begin{corollary}
Let $G$ and $H$ be finite groups acting on each other compatibly. Suppose that $G$ is one of the following groups:
\begin{itemize}
  \item[(i)] metacyclic group;
  \item[(ii)] symmetric group of $n$ elements for $n=5$ or $n\geq 8$;
  \item[(iii)] simple group with $M(G)=0$ or $M(G)=\mathbb{Z}/2\mathbb{Z}$.
\end{itemize}
 Then $B_0(G\otimes H)=0$.
\end{corollary}

\begin{proof}
Consider the central extension of finite groups $1\to C\to G\otimes H \to D_H(G)\to 1$.
We know that $D_H(G)$ is a normal subgroup of $G$. Therefore if $G$ is metacyclic group, then $D_H(G)$ is also metacyclic.
If $G=S_n$ for $n\geq 5$, then $D_H(G)$ might be either trivial, or  the alternating group $A_n$, or $S_n$. If $G$ is simple, then $D_H(G)$ might
be either trivial or $G$. Hence in any case, $D_H(G)$ satisfies the requirements of the previous proposition.
\end{proof}

\begin{remark}
There are many finite simple groups $G$ with $M(G)=\mathbb{Z}/2\mathbb{Z}$ or $M(G)=0$ (see \cite{GLS}).
 Except the Chevalley and Steinberg groups and a few other exceptions, most of the finite simple groups have Schur multiplier of order at most 2.
\end{remark}

\section*{Acknowledgement}
This work was partially supported by Ministerio de Econom\'ia y Competitividad (Spain),
grant MTM2013-43687-P (European FEDER support included), and by Xunta de Galicia, grant GRC2013-045 (European FEDER support included).
The first and the third author would like to thank the Department of Algebra at the University of Santiago de Compostela for the hospitality provided during the preparation of this article.


\end{document}